\documentclass{amsart}
\usepackage{amssymb}
\usepackage{color}

\newtheorem{theorem}{Theorem}[section]
\newtheorem{corollary}{Corollary}[section]

\newtheorem{lemma}[theorem]{Lemma}

\theoremstyle{definition}

\newtheorem{remark}{Remark}[section]

 \numberwithin{equation}{section}
\begin{document}
\title{On existence of PI-exponents of unital algebras}
\author[D.D. Repov\v s and M.V. Zaicev]
{Du\v san D. Repov\v s and Mikhail V. Zaicev}
\address{Du\v san D. Repov\v s \\Faculty of Education, and
Faculty of  Mathematics and Physics, University of Ljubljana \&
Institute of mathematics, Physics  and Mechanics, Ljubljana, 1000, Slovenia}
\email{dusan.repovs@guest.arnes.si}
\address{Mikhail V. Zaicev \\Department of Algebra\\ Faculty of Mathematics and
Mechanics\\  Moscow State University \\ Moscow,119992, Russia}
\email{zaicevmv@mail.ru}
\thanks{The first author was supported by the Slovenian Research Agency
grants P1-0292, J1-8131, N1-0114, N1-0083, and N1-0064. 
The second author was supported by the Russian Science Foundation, 
grant 16-11-10013$\Pi$}

\begin{abstract} We construct a family of unital  non-associative algebras
$\{T_\alpha\vert~ 2<\alpha\in\mathbb R\}$ such that $\underline{exp}(T_\alpha)=2$,
whereas $\alpha\le\overline{exp}(T_\alpha)\le\alpha+1$. In particular, it follows 
that ordinary PI-exponent of codimension growth of algebra $T_\alpha$  does not exist
for any $\alpha> 2$. This is the first example of a unital
algebra whose PI-exponent does not exist.
\end{abstract}

\maketitle

\section{Introduction}
We consider numerical invariants associated with polynomial identities
of algebras over a field of characteristic zero. Given an algebra $A$,
one can construct a sequence of non-negative integers $\{c_n(A)\},
n=1,2,\ldots$, called the {\it codimensions} of $A$, which is an important 
numerical characteristic of identical relations of $A$. In general, the
sequence $\{c_n(A)\}$ grows faster than $n!$. However, there is a wide class of 
algebras with exponentially bounded codimension growth. This class includes
all associative PI-algebras \cite{r2}, all finite-dimensional algebras
\cite{r2}, Kac-Moody algebras \cite{r3}, infinite-dimensional simple Lie
algebras of Cartan type \cite{r4}, and many others. If the sequence
$\{c_n(A)\}$ is exponentially bounded then the following natural question arises:
does the limit
\begin{equation}\label{eqi1}
\lim_{n\rightarrow\infty}\root n \of{c_n(A)}
\end{equation}
exist and what are its possible values? In case of existence, the limit
(\ref{eqi1}) is called the PI-{\it exponent} {\it of} $A$, denoted as $exp(A)$. At the end of 1980's, Amitsur
conjectured that for any associative PI-algebra, the limit (\ref{eqi1})
exists and is a non-negative integer. Amitsur's conjecture was confirmed
in \cite{r5,r6}. Later, Amitsur's conjecture was also confirmed for
finite-dimensional Lie and Jordan algebras \cite{r8,r7}. Existence 
of $exp(A)$ was also proved for all finite-dimensional simple algebras \cite{r9}
and many others.

Nevertheless, the answer to Amitsur's question in the general case  is negative:
 a counterexample was presented in \cite{r18}. Namely, for any real
$\alpha>1$, an algebra $R_\alpha$ was constructed such that the lower
limit of $\root n \of{c_n(A)}$ is equal to $1$,  whereas the upper limit is equal 
to $\alpha$. It now looks natural to describe classes of algebras in which
for any algebra $A$, its PI-exponent $exp(A)$ exists. One of the candidates
is the class of all finite-dimensional algebras. Another one is the class of
so-called special Lie algebras. The next interesting class consists of unital
algebras, it contains in particular, all algebras with an external unit. Given
an algebra $A$, we denote by $A^\sharp$ the algebra obtained from $A$ by
adjoining the external unit. There is a number of papers where the existence
of $exp(A^\sharp)$ has been proved, provided that $exp(A)$ exists \cite{r12,r13,
r11}. Moreover, in all these cases, $exp(A^\sharp)=exp(A)+1$.

The main goal of the present paper is to construct a series of unital algebras
such that $exp(A)$ does not exist, although the sequence $\{c_n(A)$ is exponentially
bounded (see Theorem \ref{t1} and Corollary \ref{c1} below). All details about polynomial
identities and their numerical characteristics can be found in \cite{r14,r15,r16}.

\section{Definitions and preliminary structures}

Let $A$ be an algebra over a field $F$ and let $F\{X\}$ be a free $F$-algebra with
an infinite set $X$ of free generators. The set $Id(A)\subset F\{X\}$ of all
identities of $A$ forms an ideal of $F\{X\}$. Denote by $P_n=P_n(x_1,\ldots,x_n)$ the
subspace of $F\{X\}$ of all multilinear polynomials on $x_1,\ldots,x_n\in X$. Then
$P_n\cap Id(A)$ is actually the set of all multilinear identities of $A$ of degree $n$.
An important numerical characteristic of $Id(A)$ is the sequence of non-negative integers
$\{c_n(A)\},n=1,2,\ldots~,$ where
$$
c_n(A)=\dim\frac{P_n}{P_n\cap Id(A)}.
$$
If the sequence $\{c_n(A)\}$ is exponentially bounded, then the lower and the upper
PI-exponents of $A$, defined as follows
$$
\underline{exp}(A)=\liminf_{n\rightarrow\infty} \root n \of{c_n(A)}, \quad
\overline{exp}(A)=\limsup_{n\rightarrow\infty} \root n \of{c_n(A)},
$$
are well-defined. An existence of ordinary PI-exponent (\ref{eqi1}) is equivalent to the
equality $\underline{exp}(A)=\overline{exp}(A)$.

In  \cite{r18}, an algebra $R=R(\alpha)$ such that
$\underline{exp}(R)=1,\overline{exp}(R)=\alpha$, was constructed for any real $\alpha>0$. Slightly modifying the
construction from \cite{r18}, we want to get for any real $\alpha>2$, an algebra $R_\alpha$
with $\underline{exp}(R_\alpha)^\sharp=2$ and $\alpha\le \overline{exp}(R^\sharp)\le\alpha+1$. 

Clearly, polynomial identities of $A^\sharp$ strongly depend on the identities of $A$.
In particular, we make the following observation. Note that if $f=f(x_1,\ldots, x_n)$ is
a multilinear polynomial from $F\{X\}$ then $f(1+x_1,\ldots,1+x_n)\in F\{X\}^\sharp$ is
the sum
\begin{equation}\label{eq1}
f=\sum f_{i_1,\ldots,i_k},\quad \{i_1,\ldots,i_k\}\subseteq\{1,\ldots,n\},~
0\le k\le n,
\end{equation}
where $f_{i_1,\ldots,i_k}$ is a multilinear polynomial on
$x_{i_1},\ldots,x_{i_k}$ obtained from $f$ by replacing all $x_j,j\ne i_1,\ldots,i_k$
with $1$.

\begin{remark}\label{rm1}
A multilinear polynomial $f=f(x_1,\ldots,x_n)$ is an identity of $A^\sharp$ if and only
if all of its components $f_{i_1,\ldots,i_k}$ on the left hand side of (\ref{eq1}) are
identities of $A$.
\end{remark}

 The next statement
  easily follows
  from Remark \ref{rm1}.

\begin{remark}\label{rm2}
Suppose that an algebra $A$ satisfies all multilinear identities of an algebra $B$ of
degree $\deg f=k\le n$ for some fixed $n$. Then $A^\sharp$ satisfies all identities
of $B^\sharp$ of degree $k\le n$.
\end{remark}

Using results of \cite{r17}, we obtain the following inequalities.

\begin{lemma}\label{l1} (\cite[Theorem 2]{r17})
Let $A$ be an algebra with an exponentally bounded codimension growth. Then
$\overline{exp}(A^\sharp)\le \overline{exp}(A)+1$.\hfill $\Box$
\end{lemma}

\begin{lemma}\label{l2} (\cite[Theorem 3]{r17})
Let $A$ be an algebra with an exponentally bounded codimension growth satisfying the
identity (\ref{eq2}). Then $\underline{exp}(A^\sharp)\ge \underline{exp}(A)+1$.\hfill $\Box$
\end{lemma}

Given an integer $T\ge 2$, we define an infinite-dimensional algebra $B_T$ by its basis
$$
\{a,b,z_1^i,\ldots,z_T^i\vert~i=1,2,\ldots\}
$$
and by the multiplication table

$$
z^i_j a =\left\{
               \begin{array}{l}
         z^i_{j+1} \quad \rm{if} \quad j\le T-1,    \\
          0 \qquad \rm{if} \quad j=T
               \end{array}
             \right.
$$
for all $i\ge 1$ and 
$$
z^i_Tb=z_1^{i+1},\quad i\ge 1.
$$
All other products of basis elements are equal to zero. Clearly, algebra $B_T$ is right 
nilpotent of class 3, that is
\begin{equation}\label{eq2}
x_1(x_2x_3)\equiv 0
\end{equation}
is an identity of $B_T$. Due to (\ref{eq2}), any nonzero product of elements of $B_T$
must be left-normed. Therefore we omit brackets in the left-normed products and write
$(y_1y_2)y_3=y_1y_2y_3$ and $(y_1\cdots y_k)y_{k+1}=y_1\cdots y_{k+1}$ if $k\ge 3$. 

We will use the following properties of algebra $B_T$.

\begin{lemma}\label{l3} (\cite[Lema 2.1]{r18})
Let $n\le T$. Then $c_n(B_T)\le 2n^3$.\hfill $\Box$
\end{lemma}

\begin{lemma}\label{l4} (\cite[Lema 2.2]{r18})
Let $n= kT+1$. Then 
$$
c_n(B_T)\ge k!=\left(\frac{N-1}{T}\right)!.
$$\hfill $\Box$
\end{lemma}

\begin{lemma}\label{l5} (\cite[Lema 2.3]{r18})
Any multilinear identity $f=f(x_1,\ldots,x_n)$ of degree $n\le T$ of
algebra $B_T$ is an identity of $B_{T+1}$.\hfill $\Box$
\end{lemma}
 
Let $F[\theta]$ be a polynomial ring over $F$ on one indeterminate $\theta$
and let $F[\theta]_0$ be its subring of all polynomials without free term. Denote 
by $Q_N$ the quotient algebra
$$
Q_N=\frac{F[\theta]_0}{(Q^{N+1})},
$$
where $(Q^{N+1})$ is an ideal of $F[\theta]$ generated by $Q^{N+1}$. Fix an 
infinite sequence of integers $T_1<N_1<T_2<N_2\ldots$ and consider the algebra
\begin{equation}\label {eq3}
R=B(T_1,N_1)\oplus B(T_2,N_2)\oplus\cdots\, ,
\end{equation}
where $B(T,N)=B_T\otimes Q_N$.

Let $R$ be an algebra of the type (\ref{eq3}). Then the following lemma holds.

\begin{lemma}\label{l6}
For any $i\ge 1$, the following equalities hold:
\begin{itemize}
\item[(a)]
if $T_i\le n\le N_i$ then
$$
P_n\cap Id(R)=P_n\cap Id(B(T_i,N_i)\oplus B(T_{i+1},N_{i+1}))=
P_n\cap Id(B_{T_i}\oplus B_{T_{i+1}});
$$
\item[(b)]
if $N_i <n\le T_{i+1}$ then
$$
P_n\cap Id(R)=P_n\cap Id(B(T_{i+1},N_{i+1}))=
P_n\cap( Id(B_{T_{i+1}})).
$$
\end{itemize}
\end{lemma}

\begin{proof}
This follows immediately from the equality $B(T_i,N_i)^{N_i+1}=0$ and from Lemma \ref{l5}.
\end{proof}

The folowing remark is obvious.

\begin{remark}\label{rm3}
Ler $R$ be an algebra  of type (\ref{eq3}). Then
$$
Id(R^\sharp)=Id(B(T_1,N_1)^\sharp\oplus B(T_2,N_2)^\sharp\oplus\cdots ).
$$\hfill $\Box$
\end{remark} 

\section{The main result}

\begin{theorem}\label{t1}
For any real $\alpha>1$, there exists an algebra $R_\alpha$ with  $\underline{exp}(R_\alpha)=1,
\overline{exp}(R_\alpha)=\alpha$ such that $\underline{exp}(R_\alpha^\sharp)=2$ and
$\alpha\le\overline{exp}(R_\alpha^\sharp)\le\alpha+1$.
\end{theorem}
\begin{proof}
Note that 
\begin{equation}\label{eq4}
c_n(A)\le nc_{n-1}(A)
\end{equation}
for any algebra $A$ satisfying (\ref{eq2}). We will construct $R_\alpha$ of type (\ref{eq3})
by a special choice of the sequence $T_1,N_1,T_2,N_2,\ldots$ depending on $\alpha$. First, 
choose $T_1$ such that
\begin{equation}\label{eq5}
2m^3<\alpha^m
\end{equation}
for all $m\ge T_1$. By Lemma \ref{l4}, algebra $B_{T_1}$ has an overexponential codimenson 
growth. Hence there exists $N_1> T_1$ such that
$$
c_n(B_{T_1})<\alpha^n \quad \hbox{for~all} \quad n\le N_1-1
\quad \hbox{and} \quad c_{N_1}(B_{T_1}) \ge \alpha^{N_1}.
$$

Consider an arbitrary $n>N_1$. By Remark \ref{rm1}, we have
$$
c_n(R^\sharp)\le \sum_{k=0}^n {n\choose k}c_k(R)=\Sigma_1'+\Sigma_2',
$$
where
$$
\Sigma_1'=\sum_{k=0}^{N_1} {n\choose k}c_k(R),\quad
\Sigma_2'=\sum_{k=N_1+1}^{n} {n\choose k}c_k(R).
$$
By Lemma \ref{l6}, we have $\Sigma_1'+\Sigma_2'\le \Sigma_1+\Sigma_2$, where
$$
\Sigma_1=\sum_{k=0}^{N_1} {n\choose k}c_k(B_{T_1}),\quad
\Sigma_2=\sum_{k=0}^{n} {n\choose k}c_k(B_{T_2}).
$$
Then for any $T_2>N_1$, an upper bound for $\Sigma_2$ is
\begin{equation}\label{eq6}
\Sigma_2\le\sum_{k=0}^n {n\choose k} 2k^3\le 2n^3 \sum_{k=0}^n {n\choose k}=2n^3 2^n,
\end{equation}
which follows from (\ref{eq5}), provided that $n\le T_2$.

Let us find an upper bound for $\Sigma_1$ assuming that $n$ is sufficiently large. 
Clearly,
\begin{equation}\label{eq6a}
\Sigma_1\le N_1\alpha^{N_1}\sum_{k=0}^{N_1} {n\choose k} 
\end{equation}
which follows from the choice of $N_1$, relation (\ref{eq4}), and the equality
$B(T_1,N_1)^n=0$ for all $n\ge N_1+1$. Since $N_1\alpha^{N_1}$ is a constant for fixed
$N_1$, we only need to estimate the sum of binomial coefficients.

From the Stirling formula
$$
m! =\sqrt{2\pi m} (\frac{m}{e})^m e^{\frac{1}{12m+\theta_m}},\quad 0<\theta_m<1,
$$
it follows that
\begin{equation}\label{eq7}
{n\choose k}\le \sqrt{\frac{n}{k(n-k)}}\cdot\frac{n^n}{k^k(n-k)^{n-k}}.
\end{equation}

Now we define the function $\Phi:[0;1]\to\mathbb R$ by setting
$$
\Phi(x)=\frac{1}{x^x(1-x)^{1-x}}.
$$
It is not difficult to show that $\Phi$ increases on $[0;1/2]$, $\Phi(0)=1$,
and $\Phi(x)\le 2$ on $[0;1]$. In terms of the function $\Phi$ we rewrite (\ref{eq7}) as
\begin{equation}\label{eq8}
{n\choose k}\le \sqrt{\Phi\left(\frac{k}{n}\right)}\cdot\Phi\left(\frac{k}{n}\right)^n <
2\Phi\left(\frac{k}{n}\right)^n\le 2\Phi\left(\frac{N_1}{n}\right)^n
\end{equation}
provided that $n> 2N_1$. Now (\ref{eq6a}) and (\ref{eq8}) together with (\ref{eq6})
imply
$$
\Sigma_1\le 2N_1\alpha^{N_1}(N_1+1)\Phi\left(\frac{N_1}{n}\right)^n,
\quad \Sigma_2\le 2n^3 2^n.
$$

Since
$$
\lim_{n\to\infty}\Phi\left(\frac{N_1}{n} \right)^n=1
$$
and $\Phi(x)$ increases on $(0;1/2]$, there exists $n>2N_1$ such that
\begin{equation}\label{eq9}
2N_1(N_1+1)\alpha^{N_1}\Phi\left(\frac{N_1}{n} \right)^n+2n^3 2^n < (2+\frac{1}{2})^n.
\end{equation}

Now we take $T_2$ to be equal to the minimal $n>2N_1$ satisfying (\ref{eq9}). Note that for
such $T_2$ we have
$$
c_n(R^\sharp)<(2+\frac{1}{2})^n
$$
for $n=T_2$.

As soon as $T_2$ is choosen, we can take $N_2$ as the minimal $n$ such that $c_n(B_{T_2})
\ge \alpha^n$. Then again, $c_m(R)<m
\alpha^m$ if $m<N_2$. Repeating this procedure, we can
construct an infinite chain $T_1<N_1<T_2<N_2\cdots~$ such that
\begin{equation}\label{eq9o}
c_n(R)<\alpha^n+2n^3
\end{equation}
for all $n\ne N_1,N_2,\ldots$,
\begin{equation}\label{eq9a}
\alpha^n\le c_n(R)<\alpha^n+n(\alpha^{n-1}+2n^3)
\end{equation}
for all $n= N_1,N_2,\ldots$
and
\begin{equation}\label{eq10}
2N_j(N_j+1)\alpha^{N_j}\Phi\left(\frac{N_j}{T_{j+1}} \right)^{T_{j+1}}+
2T_{j+1}^3\cdot 2^{T_{j+1}} < (2+\frac{1}{2^j})^{T_{j+1}}
\end{equation}
for all $j=1,2,\ldots~$.

Let us denote by $R_\alpha$ the  just constructed algebra $R$ of type (\ref{eq3}).
Then (\ref{eq10}) means that
\begin{equation}\label{eq11}
c_n(R_\alpha^\sharp)<(2+\frac{1}{2^j})^n
\end{equation}
if $n=T_{j+1}, j=1,2,\ldots~$. It follows from inequality (\ref{eq11}) that
\begin{equation}\label{eq12}
\underline{exp}(R_\alpha^\sharp)\le 2.
\end{equation}
On the other hand, since $R_\alpha$ is not nilpotent, it follows that
\begin{equation}\label{eq13}
\underline{exp}(R_\alpha^\sharp)\ge 1.
\end{equation}
Since the PI-exponent of non-nilpotent algebra cannot be strictly less than $1$, 
relations (\ref{eq12}), (\ref{eq13}) and Lemma \ref{l2} imply
$$
\underline{exp}(R_\alpha)=1,~\underline{exp}(R_\alpha^\sharp)=2.
$$
Finally, relations (\ref{eq9o}), (\ref{eq9a}) imply the equality $\overline{exp}(R_\alpha)=\alpha$.
Applying Lemma \ref{l1}, we see that $\overline{exp}(R_\alpha^\sharp)\le\alpha+1$.
The inequality $\alpha=\overline{exp}(R_\alpha)\le \overline{exp}(R_\alpha^\sharp)$ is obvious,
since $R_\alpha$ is a subalgebra of $R_\alpha^\sharp$, Thus we have  completed the proof
of Theorem \ref{t1}.
\end{proof}

As a consequence of Theorem \ref{t1} we get an infinite family of unital algebras of 
exponential codimension growth without ordinary PI-exponent.

\begin{corollary}\label{c1} 
Let $\beta>2$ be an arbitrary real number. Then the ordinary PI-exponent of unital algebra
$R_\beta^\sharp$ from Theorem \ref{t1} does not exist. Moreover,
$\underline{exp}(R_\beta^\sharp)=2$, whereas $\beta\le\overline{exp}(R_\beta^\sharp)\le\beta+1$.\hfill $\Box$
\end{corollary} 

\section*{Acknowledgments} We would like to thank the referee for comments and suggestions.

\end{document}